\documentclass{amsart}

\usepackage{amsmath, amsthm,amssymb,amscd}

\usepackage{mathrsfs}
\usepackage[all,dvips]{xy}
\usepackage{enumerate}
\usepackage{float}
\floatstyle{boxed} 
\restylefloat{figure}

\usepackage{subfig}
\usepackage{pstricks-add}
\usepackage[all,dvips]{xy}

\usepackage{hyperref}

\newcommand{\N}{\mathbb{N}}
\newcommand{\Z}{\mathbb{Z}}

\newcommand{\C}{\mathbb{C}}

\newcommand{\tree}{\mathcal{T}_d}
\newcommand{\aut}{\text{Aut}\left(\tree \right)}
\newcommand{\hil}{\ell^2 \left( \tree ^0 \right)}

\newcommand{\St}{\textrm{Stab}}
\newcommand{\Ri}{\textrm{Rist}}

\newcommand{\fond}{C^{*}_{\rho}\left( G \right)}

\providecommand{\norm}[1]{\lVert#1\rVert} 

\providecommand{\norminfty}[1]{\lVert#1\rVert_{_{\infty}}}

\def\Ker{\mathrm{Ker}}
\def\Hom{\mathrm{Hom}}
\def\Irr{\mathrm{Irr}}

\def\surj{\twoheadrightarrow}

\def\into{\hookrightarrow}
\def\onto{\twoheadrightarrow}

\theoremstyle{plain} 
\newtheorem{theo}{Theorem}[section]
\newtheorem{cor}[theo]{Corollary}
\newtheorem{prop}[theo]{Proposition}
\newtheorem{prop-defi}[theo]{Proposition-Definition}
\newtheorem{lem}[theo]{Lemma}
\newtheorem*{theoetoile}{Theorem} 

\newtheorem*{propetoile}{Proposition}

\newtheorem{proper}[theo]{Property}

\theoremstyle{definition} 
\newtheorem{defi}[theo]{Definition}
\newtheorem{notation}[theo]{Notation}

\newtheorem*{defitoile}{Definition}

\theoremstyle{remark}
\newtheorem{ex}[theo]{Example}
\newtheorem{remarque}[theo]{Remark}

\numberwithin{equation}{section}

\begin{document}

\title{Fundamental $C^*$-algebras associated to automata groups.}

\maketitle








\centerline{\footnotesize{by \textsc{Jean-Fran\c{c}ois Planchat}}{\footnote{J.-F. PLANCHAT, Mathematisches Institut, Georg-August Universit\"{a}t G\"{o}ttingen. \url{planchat@uni-math.gwdg.de}}}}

\begin{abstract} 
We propose to study some properties of the $C^*$-algebra naturally built out of the fundamental action that an automaton group $G$ admits on a regular rooted trees $\tree$. 
\end{abstract}

\section{Introduction}
This paper aims to study some $C^*$-algebras associated to groups generated by automata. We will mainly focus on the \emph{fundamental} $C^*$-algebra of such a group $G$, that is the one naturally built out of the action of $G$ on a regular rooted tree $\tree$ that defines it as an automaton group. The paper is organised as follows. The first section aims to recall some basic definitions and properties of regular rooted trees and their group of automorphisms, and to fix some notation which will be used in the sequel. In Section 2, we introduce the fundamental $C^*$-algebra of a automaton group and study some general properties of it. For instance, we will be interested in spectrum approximation for non-normal elements in the von Neumann algebra generated by this fundamental $C^*$-algebra. This problem is the one that motivated I. Marin and the author to write the Appendix. Section 3 aims to extend the study initiated in \cite{Moi11} of the relation between the fundamental and the regular representation for 
an automaton group. The last one focus on groups generated by \emph{reversible} automata. In this case, the self-similar structure of $G$ allows us to get results on its reduced $C^*$-algebra $C^*_{\lambda}\left(G\right)$, and on certain $C^*$-algebras considered by V. Nekrashevych in \cite{nek09}.

\section{Groups generated by automata}
We start by recalling some definitions on rooted trees and groups generated by automata. For a general presentation on the subject, we refer to \cite{bourbakiandrzej,livrenek}. 
Let $d\geq 1$ be an integer and $\tree$ be the \emph{$d$-regular rooted tree}. That is, $\tree$ is the tree in which all vertices have degree $d+1$, except one which has degree $d$. This last vertex is called the \emph{root} of $\tree$ and is denoted by $\varnothing$. If we let $X=\left\lbrace 1,\dots,d \right\rbrace $, then $\tree$ is the Cayley graph of the free monoid $X^*$ generated by $X$. Therefore, the vertex set $\tree^0$ can be identified with $X^*$. In particular, each vertex $v\in \tree^0$ induces an isometry on $\tree$ by left translation. We denote by $v\tree$ the image of $\tree$ under $v$ and call it the \emph{subtree rooted at $v$}: 

\begin{notation}
\[v\tree := \left\lbrace x \in \tree \ | \ v \in \left[ \varnothing,x\right] \right\rbrace \] 
where $\left[ a,b\right] $ is the segment between $a$ and $b$.    
\end{notation}

Let us also introduce, for all non-negative integer $n$, the \emph{$n$-th level} $L_n$ which consists of all vertices whose distance to the root is $n$:

\begin{notation}
$L_n := \left\lbrace v\in \tree^0 \ | \ \textrm{dist}\left(\varnothing,v \right)=n \right\rbrace $ 
\end{notation}
 
We refer to Fig. \ref{figarbre} for a less rigourous but more visual presentation of these definitions.

\begin{figure}[ht]
\begin{center}
\begin{pspicture}(6,7)
\put(3.2,6){$\mbox{\small $\varnothing$}$}

\psline(3,6)(1.5,5) \psline(3,6)(2,5) \psline[linestyle=dotted](2.3,5)(2.7,5) \psline[linestyle=dotted](3.2,5)(4.1,5) \psline(3,6)(4.5,5) 
\put(1.38,5){$\mbox{\tiny $1$}$} \put(1.88,5){$\mbox{\tiny $2$}$} \put(4.62,5){$\mbox{\tiny $d$}$}

\psline(1.5,5)(0.5,4) \psline(1.5,5)(0.8,4) \psline[linestyle=dotted](1.1,4)(1.3,4) \psline(1.5,5)(1.7,4) 
\put(.30,4){$\mbox{\tiny $11$}$} \put(1.8,4){$\mbox{\tiny $1 d$}$}
\psline(4.5,5)(4.3,4) \psline(4.5,5)(4.6,4) \psline[linestyle=dotted](4.9,4)(5.2,4) \psline(4.5,5)(5.5,4) 
\put(3.90,4){$\mbox{\tiny $d 1$}$} \put(5.5,4){$\mbox{\tiny $d d$}$}

\psline[linestyle=dotted](0.5,4)(-0.3,3) \psline[linestyle=dotted](5.5,4)(6.3,3) 

\put(3.5,3){$\mbox{\tiny $v$}$}
\psline(3.5,3)(3,2) \psline(3.5,3)(3.2,2) \psline[linestyle=dotted](3.5,2)(3.7,2) \psline(3.5,3)(4,2)
\psline[linestyle=dotted](3,2)(2.5,1) \psline[linestyle=dotted](4,2)(4.5,1)
\psline[linestyle=dashed](3,6)(3,5) \psline[linestyle=dashed](3,5)(3.2,4) \psline[linestyle=dashed](3.2,4)(3.5,3)

\psframe[framearc=1,linestyle=dotted](0.9,4.85)(5.1,5.15) \put(5.5,5){$\mbox{\tiny $L_1$}$}
\psframe[framearc=1,linestyle=dotted](0,3.85)(6.2,4.15) \put(6.3,4){$\mbox{\tiny $L_2$}$}

\psparabola[linestyle=dotted](2,0.5)(3.5,3.2) \put(4.5,2.5){$\mbox{\tiny $v \tree $}$}

\end{pspicture} \end{center}\caption{$d$-regular rooted tree $\tree$}\label{figarbre}  \end{figure}

The boundary $\partial \tree$ of $\tree$ is the set of infinite geodesic paths starting from the root $\varnothing$. This set is a Cantor set; if we denote by $\partial \left(v\tree\right)$ the subset of geodesics going through the vertex $v\in\tree^0$, then
\[
 \mathscr F :=\left\lbrace \partial \left(v\tree\right) \ | \ v\in \tree^0\right\rbrace 
\]
is a fundamental system of open sets in $\partial \tree$. Let us denote by $\mu_n$ the uniform probability measure on the $n$-th level $L_n$. The limit $\mu$ of the these measures is a probability measure on $\partial \tree$ characterized by 
\[
 \mu \left( \partial \left(v\tree\right)\right)= \frac{1}{d^{\textrm{L}\left(v\right)}}
\]
where $\textrm{L}\left(v\right) $ is the level of the vertex $v$, i.e. $\textrm{L}\left(v\right)=\textrm{dist}\left( v,\varnothing\right)$.

Now, let $\aut $ be the group of the automorphisms of $\tree$. Each element $g$ in $\aut$ fixes the root and for all $n$, $g$ preserves the level $L_n$, as well as the finite set $\left\lbrace v\tree \ | \ v \in L_n\right\rbrace $ of subtrees rooted at its vertices. In particular, $\aut$ acts faithfully on $\partial \tree$ by measure preserving homeomorphisms. 

The self-similar structure of $\tree$ implies that the group $\aut $ admits a natural decomposition in terms of the automorphisms group of subtrees. 

More precisely, $g$ induces a permutation $g_1$ on the set $L_1$, as well as an isomorphism $\varphi_{g(v)}(g)$ from $v\tree$ onto $g(v)\tree$, for every vertex $v$ in $L_1$. These two subtrees are canonically isomorphic to $\tree$;  therefore, $\varphi_{g(v)}(g)$ can be seen as an element of $\aut $. 

It is easy to see that this data completely determines the action of $g$ on $\tree $. In fact, we have the following decomposition:
\[\begin{array}{ccc}
\aut   & \overset{\Phi}{\longrightarrow}  & \left( \aut  \times \dots \times \aut  \right)  \rtimes \mathfrak{S}_{d}   \\
 g  &  \longrightarrow   &  ( \varphi_{1}(g), \dots, \varphi_{d}(g)) . g_1 
\end{array}
\]
where $\mathfrak{S}_{d}$ denotes the symetric group on the set of $d$ elements $L_1$. Its action on $\left( \text{Aut}\left(\tree \right) \times \dots \times \text{Aut}\left(\tree \right) \right) $ is the permutation of the coordinates. 

\begin{defi}
The isomorphism $\Phi$ is called the \emph{recursion isomorphism}. 
\end{defi}

Generalizing further, we denote by $\Phi^{(n)}$ the decomposition of $\text{Aut}\left(\tree \right) $ with respect to the level $L_n$:
\begin{equation}
 \begin{array}{ccc}
 \text{Aut}\left(\tree \right)  & \overset{\Phi^{(n)}}{\longrightarrow} & \left(\prod_{w \in L_n}   \aut \right)  \rtimes \text{Aut}\left(\rm T_{d,n}\right)   \\
 g  &  \longrightarrow   &  \left( \varphi_w \right)_{w\in L_n}  .\ g_n 
\end{array}
\end{equation}
where $\rm T_{d,n} $ is the restriction of $\tree$ to its $n$ first levels, and $\text{Aut}(\rm T_{d,n} )$ is the restriction of $\text{Aut}\left(\tree \right) $ to this stable subgraph $\rm T_{d,n} $. We remark that an element $g\in \text{Aut}\left(\tree \right) $ fixes the restriction $\rm T_{d,n} $ if and only if $g_n$ equals 1, which is also equivalent to $g$ fixing the $n$-th level $L_n$.

Following \cite{nek09}, we can write the relation between the actions of $\tree^0$ and $\aut$ by: 
\begin{eqnarray}\label{equationsemigroupe}
\forall g\in \aut, \ \forall w \in \tree^0, \ g \circ w = g(w)\circ \varphi_{g(w)}(g).  
\end{eqnarray}
   
We can now define the class of automata groups (note that in this article, we suppose every automaton to be finite): 

\begin{defi}
A subgroup $G$ of $\aut$ is said to be generated by a automaton (on $d$ letters) if $G$ admits a finite generating set $S$ which fulfills
\[
 \forall g\in S, \ \forall v\in L_1, \ \varphi_v \left( g\right) \in S.
\]
\end{defi}

Let us finish this preliminary section with a definition and some general notation:

\begin{defi}
 A subgroup $G$ of $\aut$ will be said to act \emph{spherically transitively} on $\tree$ if and only if its action on each level is transitive.
\end{defi}

\begin{notation}
 Let $G$ be a group acting on a set $X$. Let $g\in G$ and $A\subset X$. Then,
\[
 \St_G \left(A\right) :=\left\lbrace g\in G \ | \ \forall a \in A, \ g\left(a\right)=a \right\rbrace ,
\]
\[
 \text{Fix}_X \left(g \right) := \left\lbrace x\in X \ | \ g\left(x\right)=x\right\rbrace. 
\]

\end{notation}

\section{The fundamental $C^*$-algebras associated to an automaton group}

Let $d\geq 2$ be an integer and $G$ a subgroup of $\aut$.   

\subsection{Definitions}

Let $\rho$ be the fundamental representation of $G$ on $\ell^2 \left( \tree ^0 \right)$ induced from its action on $\tree$:
\[
\begin{array}{ccc}
G & \overset{\rho}{\longrightarrow} & \mathcal{U} \left( \hil \right)\\
g & \longrightarrow & \xi  \overset{\rho(g)}{\rightarrow} \xi(g^{-1}.) 
\end{array}
\]
where $\mathcal{U} \left( \hil \right)$ is the group of unitary operators of the Hilbert space $ \ell^2 \left( \tree ^0 \right)$. 

The map $\rho$ can be extended by linearity to a $\ast$-homomorphism from $\C G$ into the bounded operators $\mathcal{B} \left( \hil \right)$.   

\begin{defi}
The fundamental $C^*$-algebra of $G$ is
\[
 \fond := \overline{\rho \left( \mathbb{C}G \right) }.
\] 
\end{defi}

\subsection{Approximation of norm and spectrum}

Let $\ell^2 \left( L_n \right)\simeq \C^{d^n}$ be the subspace of $\hil$ of functions whose support is included in the $n$-th level $L_n$. We denote by $p_n$ the orthonormal projection onto $\ell^2 \left( L_n \right)$, i.e.
\[
 p_n: \hil \surj \ell^2 \left( L_n \right), \ p_n=p_{n}^*=p_{n}^2. 
\]

For all $n\in\N$ and $x\in\fond$, let
\[
 \mathcal R_n \left( G \right) := p_n \fond p_n , \ x_n := p_n x p_n.
\]

The group $\aut$ preserves the level $L_n$, and thus all its subgroups admit a representation $\rho_n$ on $\ell^2 \left( L_n \right)$. Clearly, $\rho_n$ is contained in $\rho$ and 
\[
\begin{array}{ccc}
\fond & \overset{\rho_n}{\longrightarrow} & \mathcal R_n \left( G \right) \\
x &  \longrightarrow   & x_n
\end{array}
\]
is a surjective $\ast$-homomorphism. 

Moreover, it is easy to see that $\rho_n$ is contained in $\rho_{n+1}$. Indeed, for all $v\in L_n$, let $C(v)$ be the set of \emph{children} of $v$, i.e. the set of vertices in $L_{n+1}$ whose distance to $v$ is 1. The group $\aut$ preserves the partition
\[
 L_{n+1}=\bigsqcup_{v\in L_n} C(v)
\]
and thus subspace
\[
 \mathrm{Sh}\left(n,n+1 \right) := \left\lbrace f\in \ell^2 \left( L_{n+1} \right) \ | \ f \textrm{ depends only on }C(v), \ v\in L_n \right\rbrace .
\]

Then, denoting by $\mathbf 1_{A} $ the characteristic function of a subset $A$, the map  
\[\begin{array}{ccc}
\ell^2(L_{n}) & \overset{Q_n}{\longrightarrow}  & \mathrm{Sh}\left(n,n+1 \right)\\
 \mathbf 1 _{ \{ z \} } &  \longrightarrow   &  \frac{1}{\sqrt{d}} \mathbf 1 _{C(z)}
\end{array}
\] 
is a $\aut$-equivariant unitary, i.e. it intertwines the representations $\rho_n$ and $\rho_{n+1}\mid_{\mathrm{Sh}\left(n,n+1 \right)}$.     

Therefore, we have surjective $\ast$-homorphisms 
\[
q_n : \mathcal R _n \left( G \right) \surj \mathcal R _{n-1} \left( G \right), \ n\in\N
\]
which make the following diagram commute : 

\begin{equation}\label{diag1}
\xymatrix{& & &\fond \ar@{->>}[dlll]_{\rho_0}\ar@{->>}[dll]^{\rho_1}\ar@{->>}[d]_{\rho_n}\ar@{-->>}[dr]\\ 
\C=\mathcal R _0 \left( G \right) & \mathcal R _1 \left(G \right) \ar@{->>}[l]^{q_1} & \dots\ar@{->>}[l] & \mathcal R _n \left(G\right) \ar@{->>}[l]^{q_n} & \dots\ar@{->>}[l] }
\end{equation}

The next proposition is a direct consequence of this diagram and the fact that $\left( \sum_{k=0}^{n}p_n \right)_{n\geq 0}$ is a sequence of projections converging strongly to the identity $Id$ in $ \mathcal{B}\left( \ell^2 \left( \tree \right) \right)$.

\begin{proper}\label{approxnorme}
For all $x\in \fond$, 
\[
\norm{x}=\sup_{n \in \N} \norm{x_n}=\lim_{n \to +\infty} \norm{x_n}.
\]
\end{proper}
 
In particular, $\fond$ is a residually finite dimensional $C^*$-algebra. 

\begin{remarque}
If for all $n \geq 1$ we remove $\rho_{n-1} $ to $\rho_n$, we get a representation $\tilde{\rho}$ weakly isomorphic to $\rho$. It is easy to see that $\tilde{\rho}$ is isomorphic to the unitary representation $\rho_{mes}$ that $\aut$ admits on $L^2 \left( \partial \tree \right)$ (see  \cite{mathese} \S 3.7). In particular, for every countable subgroup $G$ of $\aut$, the $C^*$-algebra $\fond$ is isomorphic to $C^*_{\rho_{mes}} \left( G \right) := \overline{\rho_{mes} \left( \mathbb{C}G \right) }$. 
\end{remarque}

Now, we consider:
\[
 \mathcal{S}\left(G \right) :=\left\lbrace \left(x_0,x_1,\dots\right) \ | \ x_n \in \mathcal R _n \left( G \right),\ q_{n+1}\left(x_{n+1}\right)=x_n \textrm{ and } \sup_n \norm{x_n}<\infty \right\rbrace 
\]

Each $x=\left(x_0,\dots,x_n,\dots\right)$ in $\mathcal S \left( G \right)$ acts on $\hil$ \emph{via} the operator 
\[
\tilde{x}=\sum_{n=0}^{+\infty} p_n x_n p_n, 
\]
where the convergence is understood for the strong topology. Here again, 
\[\norm{\tilde{x}}=\lim_{n \to +\infty} \norm{x_n}.\] 

Then, under this identification $\mathcal S \left( G \right) \subset \mathcal B \left(\hil\right)$, we have  
\begin{prop}
Let $G$ be a subgroup of $\aut$. Then, the von Neumann algebra generated by $\fond$ in $\mathcal{B}\left( \ell^2 \left( \tree \right) \right)$ is $\mathcal{S}\left( G \right)$.
\end{prop}

\begin{proof}
$\mathcal S (G)$ is a unital, involutive subalgebra of $\mathcal B (\hil)$. Let us first show that it is weakly closed. Let $\left( x(k) \right)_{k\geq 0}$ be a sequence of elements in $\mathcal S (G)$ that converges weakly to $x \in \mathcal B (\hil)$ :
\[\forall \xi, \eta \in \hil, \ \left< x(k)(\xi) \vert \eta \right> \underset{k \to \infty}{\longrightarrow} \left< x(\xi) \vert \eta \right> .
\]

For all $n\neq p$ and $\xi_n \in \ell^2 (L_n), \ \xi_p \in \ell^2 (L_p)$, we have $\left< x(\xi_n) \vert \xi_p \right> =0$ since $x(k)$ preserves $\ell^2 (L_n)$. Thus, $x$ leaves these subspaces invariant; we denote by $x_i$ the operator that $x$ induces on $\ell^2 (L_i)$. For all $i\in\N$ and $\xi, \eta$ in $\ell^2 (L_i) $, one has 
\[
\left< x_i(\xi) \vert \eta \right> = \left< x(\xi) \vert \eta \right> = \lim_{k \to \infty}\left< x(k)(\xi) \vert \eta \right> = \lim_{k \to \infty}\left< x(k)_i(\xi) \vert \eta \right> 
\]
which shows that $x(k)_i$ converges weakly to $x_i$. Now, since the algebras $\mathcal R _i (G)$ have finite dimension, this convergence holds for the norm topology: letting $k\to +\infty$, one sees that $x_i $ belongs to $ \mathcal R _i (G)$ and that $Q_i (x_i)$ equals $x_{i-1}$. Therefore, $x=\left( x_0,\dots,x_i,\dots \right) \in \mathcal S (G)$. 

Let us show that $\fond$ is weakly dense in $\mathcal S (G)$. Let $x=\left( x_0,\dots,x_i,\dots \right) $ in $ \mathcal S (G) $. By construction, for each $n\in\N$, there exists $x(n)\in \fond$ such that for all $ i\leq n, \  x_i = x(n)_i$. This last equality is equivalent to $x_n$ equals $x(n)_n $. It is a standard fact that if $f$ is a $\ast$-homomorphism between two $C^*$-algebras, then any element in the image of $f$ admits a preimage of same norm. Thus, applying the last remark to $x(n) \mapsto x(n)_n$, one sees that it is possible to choose the $x(n)$\textquoteright s uniformly bounded ($\norm{x(n)}=\norm{x_n} \leq \norm{x}$). Then, $\forall \xi,\eta \in \hil,$
\begin{eqnarray*}
\lvert \left< (x-x(n))\xi \vert \eta \right>\rvert &=& \lvert\sum_{i>n}\left< (x_i-x(n)_i) p_i (\xi) \vert p_i (\eta) \right> \rvert\\
 & \leq & \underbrace{\left( \norm{x}+\sup_{n} \norm{x(n)}\right) \norm{(1-p_n)\xi} \norm{(1-p_n)\eta}}_{\underset{n \to \infty}{\longrightarrow}0}
 \end{eqnarray*}
which finishes the proof. 
\end{proof}

Concerning the spectrum of elements in $\mathcal S \left( G \right)$, one has the following result.

\begin{prop}\label{spectre}
\begin{enumerate}[i.]
 \item\label{spectrei} If $x\in \mathcal S \left( G \right)$ is normal, then 
\[
 \textrm{sp}(x)=\overline{\bigcup_{n=0}^{+\infty}\textrm{sp}(x_n)}.
\]
\item\label{spectreii} Suppose $G$ is finitely generated. Then, the previous equality holds for each element $x\in  \mathcal S \left( G \right)$ if and only if $G$ is virtually abelian. 
\end{enumerate}
\end{prop}

\begin{proof}
 \begin{enumerate}[i.]
  \item Suppose that $\textrm{sp}(x) \setminus \overline{\bigcup_{n=0}^{+\infty}\textrm{sp}(x_n)} \neq \emptyset.$ This non-empty set is open in $\textrm{sp}(x)$: one can choose a continous function $f$ defined on $\textrm{sp}(x)$ such that 
\begin{itemize}
 \item[-] $f\neq 0$,
\item[-] $f(z)=0, \ \forall z \in \overline{\bigcup_{n=0}^{+\infty}\textrm{sp}(x_n)}$.
\end{itemize}
 The first property implies that $f(x)$ is non-zero. The second implies that for all $n \in \N$, $0=f(x_n)=f(x)_n$, the last equality coming from the fact that the continous functional calculus commutes with continous $\ast$-homomorphisms. Since $\norm{f(x)}=\lim_{n \to +\infty} \norm{f(x)_n} $, we get a contradiction.  
\item First, let us make some remarks on $\mathcal S \left( G \right)$. The algebras $\mathcal R _n\left( G \right)$ have finite dimension. Thus, each of them is a sum of matrix algebras:
\[
 \forall n \in \N, \ \exists d_{n,1},d_{n,2},\dots,d_{n,j_n} , \ \mathcal R _n\left( G \right) = \bigoplus_{i=1}^{j_n} M_{d_{n,i}}\left( \C \right)
\]
where the $d_{n,i}$ are non-negative integers corresponding to the dimension of the irreducible representations contained in $\rho_n$. Then, the homomorphism $q_n$ correspond to a projection onto certain factors of this decomposition. Therefore, the von Neumann algebra $\mathcal S \left( G \right)$ is a $\ell^{\infty}$ sum of matrix algebras: 
\begin{align}\label{decompvn}
\mathcal S \left( G \right) &= \int^{\oplus}_{\N} M_{d_k} \left( \C \right)\\ 
&:= \left\lbrace  \left( a_0,a_1,\dots,a_k,\dots\right) \ | \ \exists M, \ \forall k, \ a_k \in M_{d_k} \left( \C \right) \textrm{ and } \norm{a_k}\leq M\right\rbrace. \nonumber
\end{align}
Moreover, even if $a_k \neq x_n$ in general (except $a_0=x_0$), one has
\[
 \bigcup_{k=0}^{+\infty}\textrm{sp}(a_k)=\bigcup_{n=0}^{+\infty}\textrm{sp}(x_n).
\]
Since the representation $\rho$ is faithful, Theorem \ref{theoIJF} in Appendix \ref{apen} implies that the integers $d_k$ are uniformly bounded if and only if $G$ is virtually abelian. 

We are now ready to prove \ref{spectre}.\ref{spectreii}. First, suppose that $G$ is virtually abelian and let $d=\max_{k\in \N} d_k$. Let $x\in \mathcal S \left(G\right)$. It is obvious that $\overline{\bigcup_{n=0}^{+\infty}\textrm{sp}(x_n)}$ is contained in $\textrm{sp}\left( x \right)$ so let us prove the other inclusion. We write the decomposition of $x$ given by \ref{decompvn} by 
\[
x=\left(a_0,a_1,\dots \right), \ a_k \in M_{d_k} \left( \C \right). 
\] 
Let $\lambda \notin \overline{\bigcup_{k=0}^{+\infty}\textrm{sp}(a_k)}$ and set $r:= \textrm{dist}\left( \lambda , \overline{\bigcup_{k=0}^{+\infty}\textrm{sp}(a_k)} \right) >0$. For all $k \in \N$, $a_k-\lambda$ is invertible; we want to prove that \[\left( \left(a_0-\lambda\right)^{-1}, \left( a_1-\lambda \right)^{-1},\dots \right) \in \mathcal S \left( G \right)\]  
which is equivalent to the $\left( a_k-\lambda \right)^{-1} $ \textquoteright s to be uniformly bounded. For each $k\in \N$ and $M\in M_{d_k} \left( \C \right) $, let 
\[
 \norminfty{M}=\sup_{1 \leq i,j \leq d_k} \left|M_{i,j}\right|.
\]
Then, $\norm{M}\leq \sqrt{d_k}\times \norminfty{M}\leq \sqrt{d_k}\times\norm{M}$. Denoting by $\textrm{Comat}\left(M\right)$ the cofactor matrix of a square matrix $M$, one has:

\begin{align*}
 \norm{\left( a_k-\lambda \right)^{-1}} &= \left| \det \left( a_k-\lambda \right) \right|^{-1} \times \norm{\textrm{Comat}\left( a_k-\lambda \right)^t} \\
                                        &\leq r^{-d}\times \sqrt{d} \times \norminfty{\textrm{Comat}\left( a_k-\lambda \right)} \\
                                        &\leq r^{-d}\times \sqrt{d} \times \left(d-1\right)!\times \norminfty{a_k-\lambda}^{d-1}  \\
                                        &\leq r^{-d}\times \sqrt{d} \times \left(d-1\right)!\times \left(\norm{x}+\left|\lambda \right| \right)^{d-1}.
\end{align*}

Let us now prove the inverse implication: suppose $G$ is not virtually abelian so that $\sup_{k\in \N} d_k = +\infty$. For all $k \in \N$, set 
\[
 a_k = \left( \begin{matrix}
  1 & -1 & 0 & \cdots\\
  0 & 1 & -1 & \\
  0 & 0 & 1 & \ddots\\
  \vdots  & \vdots  &   &\ddots
 \end{matrix}\right) \in M_{d_k} \left( \C \right) .
\]
It is easy to see that 
\[
 a_k^{-1} = \left( \begin{matrix}
  1 & 1 & 1 & \cdots\\
  0 & 1 & 1 & \cdots\\
  0 & 0 & 1 & \cdots\\
  \vdots  & \vdots  &   &\ddots
 \end{matrix}\right)
\]
so that $\norm{a_k} \leq 2$, $\norm{a_k^{-1}}\geq \sqrt{d_k}$. The first property implies that there exists an element $x\in \mathcal S \left( G \right)$ whose decomposition is given by $x=\left(a_0,a_1,\dots \right) $ and the second implies that $x$ is not invertible. However, $0\notin \overline{\bigcup_{n=0}^{+\infty}\textrm{sp}(x_n)} = \left\lbrace 1 \right\rbrace$.
 \end{enumerate}

\end{proof}
\begin{remarque}
It is clear that \ref{spectre}.\ref{spectrei} also holds for $x\in\fond$ (if $x\in\mathcal A \subset \mathcal B$, then $\textrm{sp}_{\mathcal A}\left(x\right) =\textrm{sp}_{\mathcal B}\left(x\right)$ for an inclusion of unital $C^*$-algebras $\mathcal A \subset \mathcal B$). However, the author does not know whether \ref{spectre}.\ref{spectreii} holds for $\fond$.  
\end{remarque}

\subsection{Extension of the recursion morphism}
Recall that $\tree$ is the Cayley graph of the semi-group generated by $X=\left\lbrace1,2,\dots,d \right\rbrace $. In this section, we assume that $G$ is a \emph{self-similar} subgroup of $\aut$, i.e.
\[
 \forall g \in G, \ \forall v\in \tree^0, \ \varphi_v \left( g\right) \in G.
\]

For each $g$ in $G$, we consider the following matrix $\Phi(g)$, of size $d \times d$ and whose entries belong to the group $G$:
\begin{equation}\label{recu}
\Phi(g)=\left(
\begin{matrix}
\varphi_1 (g)\\         
 &\varphi_2 (g)\\
 & & \ddots \\
 & & &\varphi_{d} (g)       
 \end{matrix} \right) \times M_{g_1} 
\end{equation}
where $M_{g_1}$ is the matrix of the permutation $g_1 \in \mathfrak{S}_d$ that $g$ induces on $X$. 

Now, let us consider the Hilbert subspace $(1-p_0)\hil$ of the functions in $\hil$ null at the root $\varnothing$ of $\tree$. This space is the sum indexed by $v\in L_1 = X $ of the subspaces $\ell^2 \left( v\tree^0 \right)$: 
\[
 (1-p_0)\hil=\bigoplus_{v \in L_1} \ell^2 \left( v\tree^0 \right)=\bigoplus_{v=1}^{d} \ell^2 \left( v\tree^0 \right)
\]
where $\ell^2 \left( v\tree^0 \right)$ is the space of functions supported on the subtree $v\tree$ rooted at the vertex $v$ of the first level $L_1$. First, the group $G$ preserves $(1-p_0)\hil$ and therefore, $\rho$ restricts to it. Moreover, $\ell^2 \left( v\tree^0 \right)$ is isometric to $\hil$ \emph{via} the isometric shift $\tree \underset{v^{-1}}{\overset{v}{\rightleftarrows}} v\tree$. Hence, there is a homorphism 
\[\fond \longrightarrow \mathcal B \left( \bigoplus_{v=1}^{d} \ell^2 \left( \tree^0 \right)\right)=M_{d}\left( \mathcal B\left(\hil\right) \right).\]  

Since $\mathcal R _1 \left( G \right) \surj \mathcal R _0 \left( G \right)$ (see diagram \ref{diag1}), this homomorphism is injective. It is also easy to see that it coincides with $\Phi$ on $G\subset \fond$. Summarizing everything, one has:   
\begin{proper}\label{properext}
The recursion homomorphism $\Phi$ gives rise to an isometric embedding 
\[\fond \overset{\Phi}{\longrightarrow} M_{d}\left( \fond \right)\] 
defined on $G\subset \fond$ by \ref{recu}. 
\end{proper}
 
\subsection{The trace $Tr$}
For all $n\in \N$, let $tr_n$ be the usual normalized trace on the finite dimensional $C^*$-algebra $\mathcal B \left( \ell^2 \left( L_n\right)\right)=M_{d^n}\left(\C\right)$. That is, $tr_n$ is defined for all $M\in M_{d^n}\left(\C\right)$ by:
\[
 tr_n \left( M \right) := \frac{1}{d^n} \sum_{i=1}^{d^n} M_{i,i}.
\]

We remark that for an element $g$ in $\aut$, $tr_n \left( g_n \right)$ measures the fixed point set of $g_n$ acting on $L_n$:
\[
 tr_n \left( g_n \right) = \frac{1}{d^n}  \left| \textrm{Fix}_{L_n}\left(g_n\right) \right| = \mu_n \left(\textrm{Fix}_{L_n}\left(g_n\right)\right)
\]
where $\mu_n$ is the uniform probability measure on the finite set $L_n$. It is easy to see that 

\begin{equation}
 tr_n (g_n) = \frac{1}{d}\left( \sum_{i \in \textrm{Fix}_{ L_1}(g)} tr_{n-1}(\varphi_i (g)_{n-1})\right).\label{cvb}
\end{equation}

In particular, for all $g\in \aut$, the sequence $\left(tr_n\left(g_n\right)\right)_{n}$ is a decreasing sequence of reals in $\left[ 0,1\right]$, hence its limit exists. More precisely,
\[
\lim_{n \to +\infty} tr_n \left( g_n \right)= \mu\left(\textrm{Fix}_{\partial \tree}\left( g\right) \right).
\]

\begin{defi}
 For all $g\in\aut$, we define $Tr\left(g\right)=\lim_{n \to +\infty} tr_n \left( g_n \right)= \mu\left(\textrm{Fix}_{\partial \tree}\left( g\right) \right)$. 
\end{defi}

\begin{proper}\label{proptrace}
 Let $G$ be a subgroup of $\aut$. Then, the map $Tr$ extends to a continous and normalized trace on $\fond$. If moreover $G$ is self-similar, then for all $x\in\fond$ and $n\in \N$, one has
\begin{equation}\label{fe}
 Tr\left( x\right)=\left(tr_n\otimes Tr\right)(\Phi^{(n)} \left(x \right)).
\end{equation}
 
\end{proper}
\begin{proof}
First extend $Tr$ by linearity to $\rho\left( \C G \right)$. Since the linear forms $tr_n$ defined on $\mathcal B \left(\ell^2\left( L_n\right) \right)$ all have norm equal to 1, Property \ref{approxnorme} implies that 
\[
 \sup_{x\in \rho\left( \C G\right), \norm{x}=1} Tr\left( x \right) \leq 1.
\]
Thus, the linear form $Tr$ can be extended by continuity to the Banach space $\fond$ in which $\rho\left( \C G\right)$ is dense. It is obvious that $Tr$ is normalized and \ref{fe} is obtained by letting $n$ go to $+\infty$ in \ref{cvb}. 
\end{proof}

\subsection*{Automata groups}

The trace we just defined is linked to the the fixed points set of elements in $G$. If the latter is generated by an automaton, one can describe this set more precisely. Recall first the following lemma proven in \cite{grigzukrev}:
\begin{lem}[Grigorchuk-\.{Z}uk]\label{grigzuk}
 Let $G$ be a subgroup of $\aut$ generated by an automaton. Let $g\in G$ be such that its fixed points set $\textrm{Fix}_{\partial \tree}\left(g\right)$ has an empty interior. Then, $\textrm{Fix}_{\partial \tree}\left(g\right)$ has measure 0 (equivalently $Tr\left(g\right)=0$). 
\end{lem}
 
Here is a generalisation of this result:

\begin{lem}\label{lemtrace}
Let $G$ be a subgroup of $\aut$ generated by an automaton and $g\in G$. Then, there exists a partition 
\[
 \partial \tree = F\sqcup T \sqcup N
\]
of the boundary $\partial \tree$ such that: 
\begin{itemize}
 \item $F$ is open (i.e. is a union of boundary of subtrees) and for $\mu$-almost all $f\in F$, $g(f)\neq f$.
 \item $T$ is open and for all $t\in T$, $g(t)=t$.
 \item $N$ has measure 0. 
\end{itemize}
\end{lem}
\begin{proof}
Let us first define $F$ and $T$. We set 
\[
 \mathscr{T}:=\left\lbrace t\in \tree^0 \ | \  t\tree \subseteq \textrm{Fix}(g) \right\rbrace \textrm{ and } \mathscr{F}:=\left\lbrace f \in \tree^0 \ | \ \forall z \in f\tree, \ z\tree \nsubseteq \textrm{Fix}(g) \right\rbrace
\]

These two sets are union of subtrees. Let us define $\mathfrak T$ and $\mathfrak F$ by the following property:
\begin{align}
 \mathscr{T}=\bigsqcup_{t\in\mathfrak{T}} t\tree & \textrm{ and } T=\bigsqcup_{t\in\mathfrak{T}} \partial \left(t\tree\right) \nonumber\\
 \mathscr{F}=\bigsqcup_{f\in\mathfrak{F}} f\tree & \textrm{ and } F=\bigsqcup_{f\in\mathfrak{F}} \partial \left(f\tree\right) \nonumber
\end{align}

It is clear that $F$ and $T$ are open, that their intersection is empty and that for all $t\in T$, $g(t)=t$. 

Let us prove that the action of $g$ on $F$ is almost free:
\begin{align}
 \mu\left(\textrm{Fix}_{\partial \tree} (g) \cap F\right) &= \mu\left(\textrm{Fix}_{\partial \tree} (g) \cap \bigsqcup_{f \in \mathfrak F} \partial \left( f\tree \right) \right) \nonumber\\
&= \sum_{f\in \mathfrak F \cap \textrm{Fix}(g)} \mu \left(\textrm{Fix}_{\partial \tree} (g) \cap \partial \left( f\tree \right) \right)\nonumber\\
&=\sum_{f\in \mathfrak F \cap \textrm{Fix}(g)} \frac{\mu \left(\textrm{Fix}_{\partial \tree} (\varphi_{f} (g))  \right)}{d^{\textrm{L}(f)}}  \nonumber
\end{align}
where $\textrm{L}(f)$ still denotes the level of the vertex $f$. Now, each term in this last sum is 0 thanks to \ref{grigzuk}. Indeed, by construction, $g$ fixes no subtree included in $\mathscr T$. The same holds \emph{a fortiori} for $\varphi_{f} (g)$ if $f\in \mathfrak F \cap \textrm{Fix}(g)$, and this is equivalent to $\textrm{Fix}_{\partial \tree} (\varphi_{f} (g)) $ having empty interior. 

We will now prove that $\mu\left(T\right)+\mu\left(F\right)=1$, which will finish the proof. Let us define for all $n\in \N$,
\[
 u_n= \vert L_n \cap \left(\mathscr{T} \cup \mathscr{F}\right)\vert .
\]

We want to prove that $v_n:=\frac{u_n}{d^n}\underset{n \to \infty}{\longrightarrow}1 $. 

\emph{Fact 1}: The sequence $\left(v_n\right)_n$ converges to a real $l\in \left[0,1\right]$. Indeed, it is an increasing sequence ($\mathscr{T}$ and $ \mathscr{F}$ are both union of subtrees) of reals in  $\left[0,1\right]$.

\emph{Fact 2}: There exists $k\in \N$ such that for all $n\in\N$,
\begin{eqnarray}
v_{kn}\geq v_{k(n-1)}+\frac{1}{d^k}\left(1-v_{k(n-1)}\right) . \label{zip}
\end{eqnarray}

To see this, let us denote by $\vert .\vert$ the word length on $G$ relative to the finite generating set which defines $G$ as an automaton group. In particular, for all $w\in \tree^0, \ \vert \varphi_w (g)\vert \leq \vert g\vert$. Then, let $k \in \N$ be such that for all $ h\in G$ with $ \vert h \vert \leq \vert g\vert $, the following holds: 
\begin{equation}\label{zipp}
\exists w \in \tree^0 \textrm{ s.t. } w\tree \subset \textrm{Fix}(h) \Longrightarrow \exists w \in \tree^0 \textrm{ s.t. } w\tree \subset \textrm{Fix}(h) \textrm{ and } \textrm{L}(w) \leq k .
 \end{equation}
 
Such a positive integer exists because the set $\left\lbrace h \ | \  \vert h\vert \leq \vert g\vert \right\rbrace $ is finite.
 
Let $N\in \N$ and $w\in L_n$ which is not in $\mathscr{T} \cup \mathscr{F} $. This means that $g(w)=w, \ \varphi_w (g) \neq 1$ and there exists $w' \in w\tree $ such that $w'\tree \subset \textrm{Fix}(g)$ (i.e. $w' \in \mathscr{T}$). According to \ref{zipp} (applied to $\varphi_w (g)$), $w'$ even exists with the additional property of having a level smaller than $N+k$. Therefore, one gets the following inequality:
\[u_{N+k} \geq \underbrace{d^k u_N}_{\textrm{ contribution of } L_N\cap \left(\mathscr{T} \cup \mathscr{F}\right)} + \underbrace{d^N -u_N}_{\textrm{ contribution of } L_N \setminus \left(\mathscr{T} \cup \mathscr{F}\right)} .
\]  

If $N:=k(n-1)$, then
\[u_{nk} \geq d^k u_{k(n-1)} + d^{k(n-1)}-u_{k(n-1)}
\]
which, dividing by $d^{kn}$, yields \ref{zip}.

Now, let $n$ tend to $\infty$ in \ref{zip} to obtain
\[\frac{1}{d^k}(1-l)\leq 0
\]
which implies that $l=1$.
\end{proof}

From this follows the unicity of $Tr$:

\begin{cor}\label{cortrace}
If $G$ is a subgroup of $\aut$ generated by an automaton, then $Tr$ is the unique continous normalized trace $\tau$ on $\fond$ for which 
\[
 \forall x\in\fond, \ \forall n\in\N, \ \tau\left( x\right)=\left(tr_n\otimes\tau\right)(\Phi^{(n)} \left(x \right)).
\]
\end{cor}
\begin{proof}
Let $\tau$ be such a trace. It is sufficient to prove that for each $g\in G$, $\tau\left(g\right)=Tr\left(g\right)$ since $\tau$ is continous and $\rho\left( \C G \right)$ is dense in $\fond$. Let $g\in G$. Let $\alpha_n$ (respectively $\beta_n$) be the number of non-zero elements (respectively of 1\textquoteright s) in the diagonal of $\Phi^{(n)}(g)$. Lemma \ref{lemtrace} says that
\begin{equation}\label{truc}
 Tr(g)=\lim_{n\to +\infty} \frac{\alpha_n}{d^n}=\lim_{n\to +\infty} \frac{\beta_n}{d^n}.
\end{equation}
(the first equality is just the definition of $Tr$). Moreover, the diagonal entries of $\Phi^{(n)}(g)$ are either 0 or elements in $G$. In particular, $\tau$ being continous, there exists $M>0$ such that for each $x$ in the diagonal of $\Phi^{(n)}(g)$, $\left|\tau(x)\right|\leq M$. Therefore,
\begin{align}
 \left| \tau\left(g\right) - Tr(g)\right|&=\left| \tau\left(\Phi^{(n)}(g)\right) - Tr(g)\right|\nonumber\\
                                             &\leq M\frac{\alpha_n-\beta_n}{d^n}+\left| \tau(1)\frac{\beta_n}{d^n}-Tr(g)\right| \nonumber
\end{align}

Since $\tau$ is assumed normalized (i.e. $\tau(1)=1$), \ref{truc} implies that the left hand side tends to 0 as $n$ goes to $+\infty$.
\end{proof}

\section{Link with the reduced $C^*$-algebra}

\subsection{Generalities}
For a discrete group $G$, we denote by $\lambda$ the regular representation of $G$, i.e. 
\[
 \lambda := G \longrightarrow \mathcal{U}\left(\ell^2 \left(G\right)\right).
\]

In \cite{mathese,Moi11}, the following general problem is studied: if $G$ is a group acting on a rooted tree $\mathcal T$, when does the induced representation $\rho$ on $\ell^2 \left( \mathcal T\right)$ weakly contain the regular representation $\lambda$? The same relation can be investigated for the representations $\rho^{\otimes n}$ coming from the diagonal action of $G$ on the product $\mathcal T \times \dots\times \mathcal T$ of $n$ copies of the rooted tree. It turns out the answer is related to the size of the stabilizers of subtrees $\St_G \left(v\mathcal T\right):= \left\lbrace g\in G \ | \ g(x)=x, \ \forall x\in v\tree \right\rbrace$.

More precisely, the following results are proved:

\begin{theo}[\cite{Moi11}]\label{theosuffisant} Let $G$ be a countable group acting faithfully on a rooted tree $\mathcal T$. 
\begin{enumerate}[i.]
 \item\label{theosuffisant1} If for every vertex $v\in \mathcal T^0$, the stabilizer $\St_G (v\mathcal T)$ is trivial, then $\lambda \prec \rho$.
 \item \label{theosuffisant2}If the set $\bigcup_{v\in \mathcal T^0} \St_G (v\mathcal T)$ is finite and has cardinality $n$, then $\lambda \prec \rho^{\otimes n}$. 
 \item \label{theosuffisant3}One always has  $\lambda \prec \bigoplus_{n=1}^{+\infty}\rho^{\otimes n} $.
\end{enumerate}
\end{theo}

It is shown that the inverse implication of Theorem \ref{theosuffisant}.\ref{theosuffisant1} is not true, unless one assumes a certain algebraic condition on $G$:

\begin{notation}\label{notA}
 A countable group $G$ is said to satisfy (A) if the normalizer $N_G \left(H \right)$ of any non-central finite group $H$ has infinite index in $G$.
\end{notation}

\begin{theo}[\cite{Moi11}]\label{theonec}
Let $G$ be a countable group satisfying (A).

 Suppose that $G$ acts spherically transitively on a rooted tree $\mathcal T$. If there exists a subtree $v\mathcal T$ whose stabilizer $\St_G \left( v\mathcal T\right)$ in $G$ is not trivial, then the $\ast$-homomorphism $\rho$ defined on $\C G$ is not injective.   
\end{theo}
 
These results obviously apply when $\mathcal T$ is a regular rooted tree $\tree$ and $G$ is generated by an automaton. The two next paragraphs aim to show that the language of automata groups is relevant for giving partial answers to certain natural questions arising from \cite{Moi11}.

\subsection{The inverse implication of Theorem \ref{theosuffisant}.\ref{theosuffisant1}: the case $d=2,3$}

As suggested by Theorem \ref{theonec}, the general strategy to study the inverse implication of Theorem \ref{theosuffisant}.\ref{theosuffisant1} is the following: if there exists a subtree $v\tree$ whose stabilizer in $G$ is not trivial, and the action of $G$ on $\tree$ is spherically transitive, then there are \emph{natural} elements $M\in \C G$ in the kernel of $\rho$. The difficulty is to prove that one of them is not 0. This cannot always be achieved, and in fact the sufficient condition in Theorem \ref{theosuffisant}.\ref{theosuffisant1} is not necessary in general (see \cite[Example 4.6]{Moi11}). When $G$ is self-similar, the existence of a subtree whose stabilizer in $G$ is not trivial clearly implies that the same holds for a subtree rooted at the first level, 
\[
 \exists v \in \tree^0, \ \St_G \left(v\tree\right) \neq \left\lbrace 1\right\rbrace \Longrightarrow \exists v \in L_1, \ \St_G \left(v\tree\right) \neq \left\lbrace 1\right\rbrace.
\]
 
For small valences, this remark allows us to study more easily the non-nullity of these elements $M\in \C G \cap \text{ker} \rho$. More precisely, we can in these cases replace the algebraic condition (A) in Theorem \ref{theonec} by conditions involving only the action of $G$ on $\tree$. 

\begin{prop}\label{propvalence2}
Let $G$ be a self-similar subgroup of $\text{Aut}\left(\mathcal T _2 \right)$. Then $\lambda \prec \rho$ if and only if $\St_G \left(v\mathcal T _2 \right) =\left\lbrace 1\right\rbrace $ for all subtree $v\mathcal T _2$.  
\end{prop}

\begin{proof} 
The \emph{if} part is implied by Theorem \ref{theosuffisant}.\ref{theosuffisant1}. Let us now prove the \emph{only if} part, assuming that $G$ is not trivial, since this case is trivial. Suppose $v\in L_1=\left\lbrace 1,2\right\rbrace $ and $g$ is a non-trivial element in $G$ fixing $v\mathcal T _2 $. Since $G$ is self-similar and non-trivial, there exists in $G$ an element $h\notin \St_G \left(L_1\right)$. Then, there are non-trivial elements $x$ and $y$ in $G$ such that: 
\[
\Phi\left(\rho\left(g\right)\right)= \left(\begin{array}{cc} x & 0 \\ 0 & 1\end{array}\right) \text{ and } \Phi\left(\rho\left(hgh^{-1}\right)\right)= \left(\begin{array}{cc} 1 & 0 \\ 0 & y\end{array}\right)
\]
or maybe the other way around. In both cases, if we let 
\[M=\left(1-g\right)\left(1-hgh^{-1}\right),\] 
one sees that $\Phi\left(\rho\left(M\right) \right)=0$ which implies that $\rho\left(M\right)=0$ since $\Phi$ is injective on $\fond$ by Property \ref{properext}. Moreover, $M\neq 0$. Indeed,                                                                                                                                                                                                                                                            
\[M=1+ghgh^{-1}-g-hgh^{-1}                                                                                                                                                                                                                                                                                     \]
so that 1 cannot be cancelled since neither $g$ nor $hgh^{-1} $ is trivial.
\end{proof}
 
\begin{prop}\label{propvalence3}
Let $G$ be a self-similar subgroup of $ \text{Aut}\left(\mathcal T _3\right)$ acting spherically transitively. Suppose there is a vertex $v \in L_1=\left\lbrace 1,2,3 \right\rbrace  $ such that the group $\St_G\left(v \mathcal T _3\right)\cap \St_G\left(L_1\right)\neq\left\lbrace 1\right\rbrace$. Then, the $\ast$-homomorphism $\rho$ defined on $\C G$ is not injective.
 \end{prop}

We will need the following lemma.

\begin{lem}\label{lemmevalence3}
Let $G$ be a self-similar subgroup of $ \text{Aut}\left(\mathcal T _3\right)$ acting spherically transitively. Let $H$ be a subgroup of $G$ in which all elements have order 2. Then, $H$ is not normal.  
\end{lem}
\begin{proof}[Proof of Lemma \ref{lemmevalence3}]
Suppose that $H$ is normal. It is also abelian since all its elements have order 2. Therefore, 
\begin{equation}\label{eqval3}
 \forall h\in H, \ \forall g\in G,\ h \text{ and } ghg^{-1} \text{commute}.
\end{equation}

Let $h\in H$ be non-trivial and $n$ the smallest non-negative integer such that the permutation $h_n$ acts non-trivially on $L_n$ (of course, $n>0$). This permutation can be decomposed into a non-trivial product of transpositions $\tau_i$ with disjoints supports:
\[h_n=\prod_{i=1 \dots k} \tau_i \] 

By the minimality of $n$, there is for each $i$ a vertex $v_i\in L_{n-1}$ such that the support of $\tau_i$ is in the first level of $v_i \mathcal T _3 $: 
\begin{eqnarray*}
\forall i=1 \dots k, \ \exists v_i \in L_{n-1}, \ x_0(i),x_1(i) \in \left\lbrace 1,2,3 \right\rbrace \textrm{ s.t. } \\
x_0(i)\neq x_1(i) \textrm{ and } \tau_i = \left( v_i x_0(i), v_i x_1(i)\right).  
\end{eqnarray*}

Let us choose an index $i$ and let $x$ be the unique element in $\left\lbrace 1,2,3 \right\rbrace \setminus \left\lbrace x_0(i),x_1(i) \right\rbrace $. Let also $g\in G$ such that $g\left(v_i x_0(i)\right)= v_i x$. In particular, $g\left(v_i\right)=v_i$. Thus $g$ preserves the first level of the subtree $v_i \mathcal T _3 $ (in which the support of $\tau_i$ is contained) which imply that the permutations $h_n$ and $g_n h_n g_n^{-1}$ can be restricted to it. By \ref{eqval3} the permutations $ h_n$ and $g_n h_n g_n^{-1}$ commute, so do these restrictions. But by construction, the last are transpositions in $\mathfrak{S}_3$ with distinct support, leading to a contradiction.   
\end{proof}

\begin{proof}[Proof of Proposition \ref{propvalence3}]

For each subset $A$ of the first level $L_1$ of $\mathcal T _3$, we define 
\[\St_G \left(A\mathcal T _3  ,1\right):= \bigcap_{v\in A}\St_G\left(v \mathcal T _3\right)\cap \St_G\left(L_1\right).\]

The assumption of the proposition implies that there exists $v\in L_1$ such that $ \St_G \left(\left\lbrace v \right\rbrace \mathcal T _3  ,1\right) \neq \left\lbrace  1\right\rbrace $. The group $G$ acts transitively on $L_1$ and therefore the same holds for all $v\in L_1$ since these groups are conjugate. 

\emph{1st case}: there exists $A \subset L_1$ with $\left| A \right| =2$ and  $\St_G \left(A\mathcal T _3  ,1\right)\neq \left\lbrace  1\right\rbrace$. The construction of a non-zero element of $\C G$ in the kernel of $\rho$ is here the same as in the proof of Proposition \ref{propvalence2}: we consider a non-trivial element $g\in \St_G \left(A\mathcal T _3  ,1\right)$ and $h\in G$ such that $A\cup h\left(A\right)=L_1$ (equivalently $h\left(A\right) \nsubseteq A $; such an $h$ exists since $G$ acts transitively on $L_1$). Now, $M:=\left(1-g\right)\left(1-hgh^{-1}\right)$ is the desired element.

\emph{2nd case}: there exists $v_1\neq v_2$ and non-trivial elements $g_1\in \St_G \left(\left\lbrace v_1 \right\rbrace \mathcal T _3  ,1\right)$, $g_2\in \St_G \left(\left\lbrace v_2 \right\rbrace \mathcal T _3  ,1\right)$ which do not commute. In this case, the commutator $\left[g_1,g_2\right]$ yields a non-trivial element in $\St_G \left(\left\lbrace v_1,v_2 \right\rbrace \mathcal T _3  ,1\right) $ and we conclude using the first case.

\emph{3rd case}: for all $i\in \left\lbrace 1,2,3 \right\rbrace = L_1$ and $g_i \in \St_G \left(\left\lbrace i \right\rbrace \mathcal T _3  ,1\right)$, the elements $g_i$ and $g_j$ commute as soon as $i\neq j$. Then, let
\[
 M\left(g_1,g_2,g_3\right):= \left(1-g_1\right)\left(1-g_2\right)\left(1-g_3\right).
\]

One has $\rho\left( M\left(g_1,g_2,g_3\right) \right)=0$ and we want to find the $g_i$\textquoteright s such that $M\left(g_1,g_2,g_3\right)$ is non-zero. 
  
If $M\left(g_1,g_2,g_3\right)=0$, then 
\[ 1-\left(g_1+g_2+g_3\right)+\left(g_1 g_2 + g_2 g_3 + g_1 g_3\right) - g_1 g_2 g_3 = 0. \]

Since the $g_i$\textquoteright s are all non-trivial, one has $g_1 g_2 g_3=1$. The same argument shows that $g_1=g_2 g_3$. Thus, $g_1$ has order 2. Using the fact that the $g_i$\textquoteright s commute, one shows in the same way that $g_2$ and $g_3$ have order 2. 

Therefore, if 
\begin{equation}\label{eqprop3}
 \forall g_1,g_2,g_3 \text{ s.t. } g_i \in \St_G \left(\left\lbrace i \right\rbrace \mathcal T _3  ,1\right), \ M\left(g_1,g_2,g_3\right) = 0,
\end{equation}
then the subgroups $\St_G \left(\left\lbrace i \right\rbrace \mathcal T _3  ,1\right) $ would consist only of order 2 elements. The same would hold for the group that their union generates since $\left[g_i,g_j\right]=1$ as soon as $i\neq j$. But this last group is normal in $G$, hence Lemma \ref{lemmevalence3} prevents \ref{eqprop3} from holding.   
\end{proof}

\subsection{The classes $\mathscr C_p^{d}$}

Let $d>1$ be an integer and $G$ a subgroup of $\aut$. Theorem \ref{theosuffisant} motivates the following definition:

\begin{defi}
For each positive integer $p$, one says that $G$ belongs to the class $\mathscr C_p ^d$ if the regular representation $\lambda$ of $G$ is weakly contained in $p$-th tensor power $\rho^{\otimes p}$ of the fundamental representation $\rho$. 
\end{defi}
The last section will give examples of automata groups in $\mathscr C_1 ^d$. Besides, it is shown in \cite{Moi11} that there are subgroups of $\aut$ which do not belong to $\mathscr C_p^{d}$, for all positive integer $p$ (these are the so-called \emph{weakly branched} subgroups of $\aut$).

Note that for all positive $p$, $\rho^{\otimes p} \prec \rho^{\otimes p+1}$ because the trivial representation is contained in $\rho$. Therefore $\mathscr C_1^{d} \subset \mathscr C_2^{d}\subset \dots \mathscr C_p^{d}\subset \dots .$ The following construction aims to exhibit examples of automata groups that distinguish some of the $\mathscr C_p^{d}$\textquoteright s.     

To simplify the exposition, we will restrict ourselves to the case $d=2$, but the following construction could be extended to any valence. Let $G$ be a subgroup of $\text{Aut}\left( \mathcal{T}_2\right)$ generated by a finite set $S=\left\lbrace g_1,\dots,g_l \right\rbrace $. Let then consider $\mathfrak T \left(G\right)$ the subgroup of $\text{Aut}\left( \mathcal{T}_2\right) $ generated by $S$ and the elements $\tilde{g_i}$ defined for all $i=1 \dots l$ by $\Phi(\tilde{g_i}) = \left(1,g_i \right)$, i.e. 
\[\mathfrak T (G)=\left\langle g_1,\dots,g_l,\Phi^{-1}\left(\left(1,g_1\right)\right),\dots,\Phi^{-1}\left(\left(1,g_l\right)\right) \right\rangle .
\]

It is easy to see that if $G$ is self-similar (resp. generated by an automaton), then $\mathfrak T \left(G\right)$ is also self-similar (resp. generated by an automaton).

Moreover, the group we get is independent of the choice of the generating set $S$ of $G$. Indeed, $\mathfrak T \left(G\right)$ is also the group generated by $G \bigcup \left\lbrace \left(1,g\right)\ | \ g\in G \right\rbrace $. Thus, for any positive integer $p$, one can consider without ambiguity the group $\mathfrak T^{(p)} \left(G\right)$ that we get by applying this construction $p$ times. 

Now, let $G$ be a finitely generated subgroup of $\text{Aut}\left( \mathcal{T}_2\right)$ fulfilling moreover the following conditions:
\begin{itemize}
 \item $G$ is self-similar,
 \item the action of $G$ on the boundary $\partial \mathcal T _2$ is essentially free,
 \item the action of $G$ on $ \mathcal T _2$ is spherically transitive.
\end{itemize}

Then, 

\begin{prop}\label{propdistinction}
For all positive integer $p$, the group $\mathfrak T^{(p)} \left(G\right)$ distinguishes the classes $\mathscr C _{2^p}^2$ and $\mathscr C _{2^p -1}^2$.
\end{prop}

We will need the following lemma:

\begin{lem}\label{lemfonal}
 Let $G$ be a finitely generated subgroup of $\aut$. 
\begin{enumerate}[i.]
 \item \label{lemfonali} Suppose there exists $n\in \N$ such that:
\begin{itemize}
 \item[-] the group $G$ acts transitively on the $n$-th level $L_n$ of $\tree$,
 \item[-] there is a vertex $v_0 \in L_n$ such that the group 
\[\text{Rist}_G \left(v_0\right) := \bigcap_{v\in L_n \setminus \left\lbrace v_0 \right\rbrace } \St_G \left(v\tree\right)   \] 
is not trivial.
\end{itemize}
Then, $G$ does not belong to the class $\mathscr C _{d^n-1}^d $.
 \item\label{lemfonalii} Suppose there are $l$ elements $\xi_1,\xi_2,\dots,\xi_l$ in the boundary $\partial \tree$ such that  
\[\bigcap_{i=1\dots l} \St_G(\xi_i) =\left\lbrace 1 \right\rbrace   \] 

Then, the group $G$ belongs to the class $\mathscr C _l ^2$. 
\end{enumerate}
\end{lem}

\begin{proof}[Proof of Lemma \ref{lemfonal}]
i. The strategy is the same that the one used in the previous subsection: we will find a non-zero element of $\C G$ in the kernel of $\rho^{\otimes n}$. Since $G$ acts transitively on $L_n$, the groups $\Ri_G \left(v\right)$ for $v\in L_n$ are all conjugate and thus, all non-trivial. Let us choose, for each $v\in L_n$, a non-trivial element $g_v$ in the group $\Ri_G \left(v\right)$. We have
\begin{equation}\label{lasteq}
  \Phi^{(n)}\left(g_v \right) = \left(1,1,\dots,1,\varphi_v \left(g_v\right),1,\dots \right)
\end{equation}
where the non-trivial element $\varphi_v \left( g_v \right)$ appears in the position corresponding to $v$. It is clear that the $g_v$'s commute. Let  
\begin{equation*}
M=\prod_{v\in L_n} \left( 1- g_v \right) \in \C G. 
\end{equation*}

Then, $M \neq 0$. Indeed, the nullity of $M$ would imply the existence of a subset $\mathcal A$ of $L_n$ (of odd cardinality) such that 
\[1=\prod_{v\in \mathcal A} g_v,\]
and this is impossible because, if $w$ is any vertex in $\mathcal A$, (\ref{lasteq}) implies that 
\[\varphi_w \left(\prod_{v\in \mathcal A}  g_v \right)=\varphi_w \left(g_w \right)\neq 1.\]

Let us show that $\rho^{\otimes d^n -1} \left( M \right)$ is 0. This is equivalent to proving that for every $d^n-1$-tuple $\left(z_1,\dots,z_{d^n -1} \right)$ consisting of elements in $\tree^0$, 

\[\rho^{\otimes d^n -1} \left( M \right) \left(\delta_{z_1}\otimes \dots \otimes\delta_{z_{d^n -1}} \right)=0 \]
where $\delta_z \in \ell^2\left(\tree^0\right)$ is the Dirac function over the vertex $z\in \tree^0$.

There is necessarily a vertex $v_0$ among the $d^n$ in $L_n$ such that the subtree $v_0\tree$ does not contain any $z_i$. By construction, 
\[\rho^{\otimes d^n -1} \left(1-g_{v_0} \right)\left(\delta_{z_1}\otimes \dots \otimes\delta_{z_{d^n -1}} \right)=0.\]

As the $g_v$'s commute, 
\[
 M=\left( \prod_{v \in L_n \setminus \left\lbrace v_0\right\rbrace } \left( 1-g_v\right) \right) (1-g_{v_0})
\]
and this implies $\rho^{\otimes d^n -1}\left( M \right) \left(\delta_{z_1}\otimes \dots \otimes\delta_{z_{d^n -1}} \right) =0 $. 

ii- we will use the following result \cite[Proposition 3.5]{Moi11}:

\begin{propetoile}\label{propsuf}
 Let $G$ be a countable group acting on a countable set $X$. Let $\rho$ be the permutational representation that $G$ then admits on $\ell^2 \left( X \right)$: $\rho (g) (\xi) (x)= \xi (g^{-1}\cdot x)$ (with $g\in G, \ \xi \in \ell^2\left( X \right) $ and $x \in X$). 

Suppose that for every finite subset $F$ of $G$ which does not contain 1, there is an element $x$ in $X$ such that $\St_G (x) \cap F $ is empty i.e.
\[
 \forall f \in F, f.x \neq x .
\]
Then $\lambda \prec \rho$.
\end{propetoile}

Let us show that the action of $G$ on the product $\tree^0 \times \cdots \times \tree^0$ of $l$ copies of $\tree^0$ satisfies the condition of this proposition. Let $F$ be a finite subset of $G$ which does not contain 1. For each $f\in F$, the set $F\cap \St_G\left(\xi_1 \times \dots \times \xi_l\right)=F \cap \bigcap_{i=1\dots l} \St_G\left(\xi_i\right)$ is empty, i.e.
\[\forall f \in F, f\left(\xi_1\times\dots \times \xi_l \right) \neq \xi_1\times\dots \times \xi_l.
 \]

Since $F$ is finite, there exists $n\in \N$ such that (denoting by $(\xi_i)_n$ the restriction to the $n$-th level of the path $\xi_i$), one has:
\[\forall f \in F, f\left((\xi_1)_n\times\dots \times (\xi_l)_n \right) \neq (\xi_1)_n\times\dots \times (\xi_l)_n,
\] 
which completes the proof.
\end{proof}

\begin{proof}[Proof of Proposition \ref{propdistinction}]

Let $p$ be a positive integer. Since $G$ is a subgroup of $\mathfrak T^{(p)} \left(G\right)$, the last acts spherically transitively on $\mathcal T _2$. Moreover, an easy induction on $p$ implies that for each vertex $w \in L_p$: 
\begin{enumerate}
 \item the group $\Ri_{\mathfrak T^{(p)} \left(G\right)}\left(w\right)$ is not trivial,
 \item for all $\alpha \in \mathfrak T^{(p)}(G)$, $\varphi_w (\alpha) \in G$.
\end{enumerate}

The first point (1) and Lemma \ref{lemfonal}.\ref{lemfonali} already imply that the group $\mathfrak T^{(p)} \left(G\right) $ does not belong to the class $\mathcal C _{2^p -1}^2$.

Let us now show that $\mathfrak T^{(p)} \left(G\right) $ belongs to $\mathcal C _{2^p}^2$. For all non-trivial element $g \in G$, the set $\textrm{Fix}_{\partial \mathcal T_2}\left(g\right)$ has measure 0. Since $G$ is countable, the union of these fixed point sets also has measure 0. Hence, we can choose for each $w\in L_p$ an element $\xi_w \in \partial \left(w \mathcal T_2\right)$ in its complement, i.e. 
\begin{equation}\label{laeq}
\xi_w \in \partial (w \mathcal T_2) \textrm{ is such that } \forall g \in G, \ \left(g\neq 1 \Longrightarrow g(\xi_w)\neq \xi_w\right).
\end{equation}

Now, let $\alpha \in \mathfrak T^{(p)} \left(G\right) $ and suppose that for all $w \in L_p$, $\alpha\left(\xi_w\right)=\xi_w $. Necessarly, $\alpha$ fixes each element in the $p-$level $L_p$ and thus, the decomposition of $\alpha$ with respect to $L_n$ is: 
\[\Phi^{(p)}\left(\alpha\right)=\left(g_{w_0}, \dots, g_{w_{2^p-1}} \right)
\]
with $g_{w_{i}} \in G$ thanks to (2). By the definition of $\alpha$, $g_{w_{i}} $ fixes $\xi_{w_i}$ for all $i=0 \dots 2^p-1$. By \ref{laeq}, all the $g_{w_i}$\textquoteright s are trivial. Therefore, $\alpha=1$ which implies that $\mathfrak T^{(p)} \left(G\right) $ satisfies the conditions of Lemma \ref{lemfonal}.\ref{lemfonalii} with $l=2^p$. Hence, $\mathfrak T^{(p)} \left(G\right) $ belongs to $\mathcal C _{2^p}^2$. 
\end{proof}

\section{The ¨free¨ case}

We are interested here in automata groups fulfilling the condition of Theorem \ref{theosuffisant}.\ref{theosuffisant1}, that is acting topologically freely on the boundary $\partial\left(\tree\right)$. These are exactly the groups generated by \emph{reversible} automata (see Chapter 2.7 in \cite{bartsilva}). Note that this class contains a lot of interesting examples such as: the lamplighter group $\left( \oplus_{n\in\Z}\Z/ 2\Z\right)\rtimes\Z$ \cite{grigzukrev}, free groups \cite{aleshlibre,vorlibre} and \cite{glasnermozes}, lattices in $p$-adic Lie groups \cite{glasnermozes} and $\Z^n \rtimes GL_n\left( \Z \right)$ \cite{brunnersidki98}.

Let $d>1$ be an integer and $G$ be a subgroup of $\aut$ generated by a reversible automaton. Lemma \ref{grigzuk} implies that the action of $G$ on $\partial \left( \tree \right)$ is essentially free, i.e. 
\begin{equation}\label{equtrssf}
\forall g\in G, \ g \neq 1 \Longrightarrow Tr(g)=0 .
\end{equation}

This means that $Tr$ coincides on $ \C G $ with the usual trace $\tau$ on $C^* _{\lambda} \left(G \right)$ defined by $\tau (g)=\left\langle g \left(\delta_1\right)  |  \delta_1\right\rangle $, where $\delta_1 \in \ell^2 \left( G \right)$ is the Dirac function over the neutral element $1\in G$. 

Moreover, Theorem \ref{theosuffisant}.\ref{theosuffisant1} implies that $\lambda$ is weakly contained in $\rho$, i.e. $C^* _{\lambda} \left(G \right)$ is a quotient of $\fond$. Thus, we have the following commutative diagram.

\begin{equation} \label{1diag}
\xymatrix{\fond \ar@{->>}[r]^{\lambda} \ar[dr]_{Tr} & C^*_{\lambda}\left( G\right) \ar[d]_{\tau}\\
&\C}
\end{equation}

Since the trace $\tau$ is faithful, one has
\begin{equation}\label{ker}
 \ker \lambda = \left\lbrace x \in \fond \ | \ Tr(x^* x)=0 \right\rbrace .
\end{equation}

\begin{remarque}
 The representation $\rho$ contains the trivial representation. In particular, \ref{ker} implies that $Tr$ is a faithful trace on $\fond$ if and only if $G$ is amenable. The author does not know how to decide the faithfulness of $Tr$ when $G$ does not act essentially freely on $\partial \left(\tree\right)$.
\end{remarque}

Here is the central result of this section.

\begin{theo}\label{theocl}
Let $G$ be a subgroup of $\aut$ generated by a reversible automaton. Then:
\begin{enumerate}[i.]
 \item \label{theocli}The recursion isomorphism $\Phi $ defined on $\C G $ extends to an isometry from $C^{*}_{\lambda}(G) $ into $M_d(C^{*}_{\lambda}(G))$:
\[ \Phi : C^{*}_{\lambda}\left(G\right) \hookrightarrow  M_d\left( C^{*}_{\lambda}\left(G\right)\right) .
\] 
 \item \label{theoclii}For all $n\in \N$, there exists a conditional expectation $E_n$ from $M_{d^n}\left(C^{*}_{\lambda}\left(G\right)\right)$ onto $C^{*}_{\lambda}\left(G\right) $, the last being identified with its image in $M_{d^n}\left(C^{*}_{\lambda}\left(G\right)\right) $ via $\Phi^{(n)}$ : 
\[E_n : M_{d^n}\left(C^{*}_{\lambda}\left(G\right)\right) \surj \Phi^{(n)}\left(C^{*}_{\lambda}\left(G\right)\right) .
\] 
\end{enumerate}
\end{theo}

\begin{ex}
 Consider $\mathbb F_3=\left\langle a,b,c\right\rangle $ the free group on three generators $a$, $b$ and $c$. Then, Theorem \ref{theocl} and the main result of \cite{vorlibre} imply that the map
\[
 a \rightarrow \left(\begin{array}{cc}
                0 & b \\ c & 0
               \end{array}\right)
, \ b \rightarrow \left(\begin{array}{cc}
                0 & c \\ b & 0
               \end{array}\right)
, \ c \rightarrow \left(\begin{array}{cc}
                a & 0 \\ 0 & a
               \end{array}\right)
\]
defines an isometry from $C^{*}_{\lambda}\left(\mathbb F_3\right) $ into $M_{2}\left(C^{*}_{\lambda}\left(\mathbb F_3\right)\right)$. Moreover,  there exists a conditional expectation from the last $C^*$-algebra onto the image of this isometry.
\end{ex}

\begin{proof} i. \ref{1diag} and Proposition \ref{proptrace} yield the following commutative diagram:
\[ 
 \xymatrix{
    \C \ar@{=}[rr]  && \C   \\
    & \fond \ar[lu]^{Tr} \ar@{^{(}->}[rr]^{\Phi} \ar@{->>}[dl]^{\lambda} && M_{d}\left(\fond\right) \ar[lu]_{Tr} \ar@{->>}[dl]^{Id \otimes \lambda}\\
    C^{*}_{\lambda}\left( G\right) \ar[uu]^{\tau} \ar@{.>}[rr]^{?}  && M_{d}\left( C^{*}_{\lambda}\left(G\right)\right) \ar[uu]^{\tau} |!{[ul];[ur]}\hole \\
    }
\]   

Therefore, 
\begin{eqnarray*}
 \Phi \left( \ker \lambda \right) & = & \Phi \left( \left\lbrace  x \in \fond \ | \  Tr(x^* x)=0\right\rbrace \right)\\
                 & =  & \left( \textrm{Im }\Phi \right) \cap \left\lbrace x \in M_d\left(\fond\right) \ | \ Tr(x^*x)=0\right\rbrace  \\
                  & = &  \left(\textrm{Im }\Phi \right) \cap  \ker \left(Id \otimes \lambda\right) .
\end{eqnarray*}

Thus, the map $\Phi$ defines an injective (hence isometric) $\ast$-homomorphism from $C^{*}_{\lambda}\left(G\right) $ into $M_{d}\left( 
C^{*}_{\lambda}\left(G\right)\right) $. 

ii. Let us fix $n\in \N$. To simplify notation, we set
\begin{itemize}
 \item[-] $\mathscr A := \Phi^{(n)}\left( C^{*}_{\lambda}\left(G\right)\right)$ and $\mathscr A ''$ the von Neumann algebra generated by $\mathscr A $ in $\mathcal B \left(\ell^2\left(G\right)\right) \otimes M_{d^n}\left(\C\right)$,
 \item[-] $\mathscr B :=M_{d^n}\left(C^{*}_{\lambda}\left(G\right)\right) $ and $\mathscr B '':=M_{d^n}\left(\mathscr L\left(G\right)\right)$ the von Neumann algebra generated by $\mathscr B $ in $\mathcal B \left(\ell^2\left(G\right)\right) \otimes M_{d^n}\left(\C\right)$.
\end{itemize}

One knows that there exists a conditional expectation $E_n$ from $\mathscr B ''$ onto $\mathscr A ''$ determined by the following equation: 
\begin{equation}\label{equadual}
 \forall a \in \mathscr A'', \ \forall b \in \mathscr B'' , \ \tau(a^* b)=\tau(a^* E_n(b)). 
\end{equation}

The space $M_{d^n}(\C G)=M_{d^n}(\C)\otimes \C G $ is generated by the elements $e_{i,j} \otimes g$ where $g$ runs over $G$ and the ordered pair $(i,j)$ runs over $\left\lbrace 1,\dots,d^n\right\rbrace^2 $ (i.e., the $(i,j)$-th entry of the matrix $e_{i,j} \otimes g$ is $g$, the others are 0). We will show that 
\begin{equation}\label{app}
E_n \left( e_{i,j} \otimes g\right) \in \Phi^{(n)}\left(\C G\right) 
\end{equation}
which will prove the second part of Theorem \ref{theocl}. Indeed, the linear map $E_n$ being bounded, \ref{app} will imply that $E_n\left( \mathscr B \right)$ is a subspace of $ \mathscr A$, and \emph{a fortiori} is $ \mathscr A$ since by definition, $E_n$ is the identity restricted to this $C^*$-algebra $\mathscr A\subset\mathscr B $.

Let $g\in G$ and $(i,j)\in\left\lbrace 1,\dots,d^n\right\rbrace^2  $. If there exists $h\in G$ such that the  $(i,j)$-th entry of the matrix $\Phi^{(n)}(h)$ is $g$, then $h$ is unique; indeed:

\begin{eqnarray*}
\exists (i,j) \textrm{ s.t. }\left( \Phi^{(n)}(h) \right)_{i,j} = \left( \Phi^{(n)}(h') \right)_{i,j} &\Longleftrightarrow& \exists v\in L_n \textrm{ s.t. } h^{-1}h' \in \St_G(v \tree)\\ 
&\Longleftrightarrow &h=h' 
\end{eqnarray*}
the last equivalence being the assumption of the theorem. Thus, we can define 
\[\tilde{E_n}(e_{i,j} \otimes g) = \left\lbrace \begin{array}{ll}
 \frac{1}{d^n}\Phi^{(n)}(h) & \textrm{if } \exists h \textrm{ s.t. }\left( \Phi^{(n)}(h) \right)_{i,j}=g, \\
 0 & \textrm{otherwise.}
\end{array} \right.  
\]

This defines a linear map $\tilde{E_n}$ on $M_{d^n}(\C G)$. In fact, $\tilde{E_n} =E_n$, which yields \ref{app}. To see this, it is sufficient to prove: 
\begin{equation}\label{eqvn}
 \forall a\in \mathscr A '', \ \tau \left(a^* \cdot (e_{i,j} \otimes g)\right)=\tau\left(a^* \cdot \tilde{E_n}(e_{i,j} \otimes g)\right)
\end{equation}
since $E_n$ is determined by \ref{equadual} and the $e_{i,j} \otimes g $\textquoteright s generate $\mathscr B ''$. 

For $h \in G$, the diagonal of the matrix $\Phi^{(n)}\left(h\right)^{*} \cdot (e_{i,j} \otimes g)$ consists only of 0\textquoteright s except maybe at the $j$-th position: 
\begin{itemize}
 \item[-] Either $\left(\Phi^{(n)}(h)^{*} \cdot (e_{i,j} \otimes g)\right)_{j,j} = 1$. This is equivalent to $\left( \Phi^{(n)}(h) \right)_{i,j} =g$ and in this case, 
\[\tau \left( \Phi^{(n)}\left(h\right)^{*} \cdot (e_{i,j} \otimes g)\right) =\frac{1}{d^n}= \tau \left( \Phi^{(n)}\left(h\right)^{*} \cdot \tilde{E_n}(e_{i,j} \otimes g)\right).
\]
 \item[-] Or $\left(\Phi^{(n)}\left(h\right)^{*} \cdot(e_{i,j} \otimes g)\right)_{j,j} = 0$ or $ \alpha $, where $\alpha\neq 1$ belongs to $G$. In this case, 
$$\tau\left(\Phi^{(n)}\left(h\right)^{*}\cdot (e_{i,j} \otimes g)\right) =0= \tau \left( \Phi^{(n)}\left(h\right)^{*} \cdot\tilde{E_n}(e_{i,j} \otimes g)\right). $$
\end{itemize}

Thus, we observe that $\tau \left(a^* \cdot (e_{i,j} \otimes g)\right)=\tau\left(a^* \cdot \tilde{E_n}(e_{i,j} \otimes g)\right) $ for each element $a$ in the set $\left\lbrace \Phi^{(n)}\left(h\right) \ | \ h \in G\right\rbrace$. Since $\mathscr A ''$ is generated by this set, we have proved \ref{eqvn}, hence the theorem. 
\end{proof}

From Theorem \ref{theocl}.\ref{theocli} follows the existence of the direct system $ \left\lbrace M_{d^n}\left(C^* _{\lambda}\left(G\right)\right),\Phi \right\rbrace $. Therefore, one can consider the $C^*$-algebra $\mathcal L _G$ limit of this system:
\[\mathcal L _G = \varinjlim_{\Phi,n} M_{d^n}\left(C^* _{\lambda}\left(G\right)\right).
\]

Moreover, Theorem \ref{theocl}.\ref{theoclii} gives
\begin{cor}
Let $G$ be a subgroup of $\aut$ generated by a reversible automaton. Then, there exists a conditional expectation $E$ from $\mathcal L _G$ onto the image of $C^* _{\lambda}\left(G\right) $ in $\mathcal L _G$. 
\end{cor}
\begin{proof}
 Let us first remark that, since the $\Phi^{(n)}$\textquoteright s are injective, the image of $C^* _{\lambda}\left(G\right) $ in $\mathcal L _G$ is actually isomorphic to $C^* _{\lambda}\left(G\right) $. For all $n\in \N$, we can define 
\[F_n : M_{d^n} \left(C^* _{\lambda}\left(G\right)\right) \stackrel{E_n}{\longrightarrow} \Phi^n \left(C^* _{\lambda}\left(G\right)\right) \stackrel{i}{\longrightarrow} \varinjlim_{n} \Phi^{(n)} \left(C^* _{\lambda}\left(G\right)\right)
\]
where $E_n$ is the conditional expectation built in \ref{theocl} and $i$ is the injection of $M_{d^n}\left(C^* _{\lambda}\left(G\right)\right)$ in $\mathcal L _G$. The map $E_n$ has norm 1 since it is a normalized conditional expectation. The same holds for $i$ and thus for $F_n$. Moreover, by the definition of $E_n$, the linear maps $F_n$\textquoteright s are consistent with the direct system $ \left\lbrace M_{d^n}\left(C^* _{\lambda}\left(G\right)\right),\Phi \right\rbrace $. Therefore, there exists a linear map $E$ from the algebraic direct limit $\tilde{\mathcal L _G}$ onto $\varinjlim_{n} \Phi^n \left(C^* _{\lambda}\left(G\right)\right) \simeq C^* _{\lambda}\left(G\right)$. It has norm 1 and it is straightforward to check that it fulfills all the axioms of a positive linear map of the $C^* _{\lambda}\left(G\right)$-$ C^* _{\lambda}\left(G\right)$-bimodule $\tilde{\mathcal L _G}$: its unique extension to the closure $\mathcal L _G$ of $\tilde{\mathcal L _G}$ is the conditional expectation we wanted.
\end{proof}

In \cite{nek09}, V. Nekrashevych defines for each self-similar subgroup of $\aut$ its \emph{universal Cuntz-Pimsner} $C^*$-algebra $\mathcal O _G$:

\begin{defitoile}[\cite{nek09}] For $G$ a self-similar subgroup of $\aut$ and $X = \left\lbrace 1,\dots,d \right\rbrace $ (i.e. the first level $L_1$ of $\tree$),  $\mathcal O _G$ is the universal $C^*$-algebra generated by a familly $\left(u_g\right)_{g\in G}$ of unitaries and a finite set $\left( S_v\right)_{v\in X}$ of partial isometries satisfying:
\begin{enumerate}
 \item for all $g,h \in G$, $u_g \cdot u_h=u_{gh}$,
 \item for all $x\in X$, $S_x^{*}S_x=1$ and $\sum_{x\in X}S_x S_x^{*} =1$,
 \item for all $g \in G$ and $x\in X$, $u_g \cdot S_x = S_{g(x)} \cdot u_{\varphi_x (g)}$.
\end{enumerate}

\end{defitoile}

V. Nekrashevych proves in particular:

\begin{theoetoile}[\cite{nek09}]
There exists a conditional expectation $\mathfrak{E}$ from $\mathcal O _G$ onto a subalgebra $\mathcal M _G$ of $\mathcal O _G$. Moreover, $\mathcal M _G$ isomorphic to the limit of the direct system $ \left\lbrace M_{d^n}\left( C^* _{\textrm{max}}\left(G\right)\right),\Phi \right\rbrace $ :
\[ \mathcal M _G = \varinjlim_{\Phi,n} M_{d^n}\left(C^* _{\textrm{max}}\left(G\right)\right).
\]
\end{theoetoile}

The author also shows that under certain conditions, the algebras $\mathcal O_G$ and $\mathcal M_G$ are nuclear if $G$ is \emph{contractible} (see \cite[Corollary 5.7.]{nek09}). 

In our case, the situation concerning the nuclearity of these algebras is the following. The $C^*$-algebra $\mathcal M _G$ clearly surjects onto $\mathcal L _G$. Thus, for a group generated by a reversible automaton, we have:
\begin{eqnarray}
\mathcal O _G \stackrel{\mathfrak{E}}{\rightarrow}\mathcal M _G \twoheadrightarrow \mathcal L _G \stackrel{E}{\rightarrow}C^* _{\lambda}\left(G\right) .
\end{eqnarray}

\begin{cor}\label{mrgou}
Let $G$ be a subgroup of $\aut$ generated by a reversible automaton. Among the $C^*$-algebras $\mathcal O _G$, $\mathcal M _G$, $\mathcal L _G$ and $C^* _{\lambda}\left(G\right)$, one is nuclear if and only if all the others are, that is if and only if $G$ is amenable. 
\end{cor}

\begin{proof}
 It is essentially a combination of known results on nuclearity for $C^*$-algebras. Let us quote them: nuclearity passes to quotient (see for instance \cite[Theorem 10.1.4. p.302]{osawabrown}), to the image of a conditional expectation (see for instance \cite[Proposition 10.1.2. p.301]{osawabrown}) and to direct limit (see for instance \cite[Theorem 10.1.5. p.302]{osawabrown}). Moreover, it is classical that $C^* _{\lambda}\left(G\right)$ is nuclear if and only if $G$ is amenable, that is if and only if $C^* _{\lambda}\left(G\right)\simeq C^* _{max}\left(G\right)$. It is easy to see that these results give all the implications of Corollary \ref{mrgou}, except $\mathcal M _G  $ nuclear $\Rightarrow \mathcal O _G$ nuclear. This is a application of the following isomorphism proved by V. Nekrashevych (\cite[Theorem 3.7.]{nek09}):
\[\left(\mathbb K \otimes \mathcal M _G \right) \rtimes \Z \simeq \mathbb K \otimes \mathcal O _G
\]
where $\mathbb K$ is the (nuclear) algebra of compact operators. If $\mathcal M _G $ is nuclear, so is $\left(\mathbb K \otimes \mathcal M _G \right) \rtimes \Z $ (see for instance \cite[Proposition 10.1.7. p.302 and Theorem 4.2.6. p.124]{osawabrown}). Therefore, the algebra $\mathbb K \otimes \mathcal O _G $ is nuclear, and $\mathcal O _G $ as well (see for instance \cite[Proposition 10.1.7. p.302]{osawabrown}).   
\end{proof}

\begin{remarque}
Note that it is still an open question whether all contractible groups are amenable. Corollary \ref{mrgou} does not bring anything new to this question, since it is already known that contractible groups acting essentially freely on $\partial \tree$ have polynomial growth (see \cite{livrenek}). 
\end{remarque}

\subsection*{Acknowledgments.} \emph{This work is part of my PhD-thesis. I want to express my gratitude to my adviser Andrzej \.{Z}uk. I also would like to thank Mikael de la Salle for usefull discussions, Suliman Albandik for having suggested the application \ref{cortrace} of Lemma \ref{lemtrace}, and Ivan Marin for having accepted to write the following appendix with me.} 

\bibliographystyle{amsalpha}
\bibliography{biblio}

\newpage

\appendix

\section{Virtually abelian groups} \label{apen}


\centerline{\footnotesize{by \textsc{Ivan Marin}}\footnote{I. MARIN, Institut de Math\'ematiques de Jussieu, Universit\'e Paris 7. \url{marin@math.jussieu.fr}} \footnotesize{and \textsc{Jean-Fran\c{c}ois Planchat}}\footnote{J.-F. PLANCHAT, Mathematisches Institut, Georg-August Universit\"{a}t G\"{o}ttingen. \url{planchat@uni-math.gwdg.de}}} 

\bigskip

Let $\Gamma$ be a finitely generated and residually finite group. We consider a decreasing sequence $\left(\Gamma_i\right)_i$ of finite index subgroups of $\Gamma$. It is easy to see that the following conditions are equivalent:

\begin{enumerate}
 \item The natural representation that $\Gamma$ admits on $\bigoplus_{i} \ell^2 \left( \Gamma/\Gamma_i \right)$ is faithful.
 \item The natural action that $\Gamma$ admits on $\bigsqcup_{i} \Gamma/\Gamma_i$ is faithful.
 \item One has 
\[\bigcap_{i} \text{Core}\left(\Gamma_i\right)=\left\lbrace  1\right\rbrace  \]
where $\text{Core}\left(\Gamma_i\right) := \bigcap_{\gamma\in \Gamma} \gamma \Gamma_i \gamma^{-1}$ is the biggest normal subgroup of $\Gamma$ contained in $\Gamma_i$.
\end{enumerate}

\begin{defi}
An F-filtration of $\Gamma$ is a decreasing sequence $\left(\Gamma_i\right)_i$ of finite index subgroups fulfilling one of the previous equivalent conditions. 
\end{defi}

\begin{theo}\label{theoIJF} Let $\Gamma$ be a finitely generated and residually finite group. The following are equivalent.
\begin{enumerate}
\item $\Gamma$ is virtually abelian.
\item $\Gamma$ is of type $I$ (see \cite{DIXM}).
\item There exists $N\in \N$ and a F-filtration $\left(\Gamma_i\right)_i$ of $\Gamma$, such that the natural representation of $\Gamma$ on $\ell^2(\Gamma/\Gamma_i)$ can be decomposed in irreducible representations of dimension at most $N$ for all $i$.
\item There exists $N\in \N$ such that, for all F-filtrations $\left(\Gamma_i\right)_i$ of $\Gamma$, the natural representation of $\Gamma$ on $\ell^2(\Gamma/\Gamma_i)$ can be decomposed in irreducible representations of dimension at most $N$ for all $i$.
\item There exists $N \in \N$ such that every finite image finite-dimensional irreducible representation of $\Gamma$ has dimension at most $N$.
\item There exists $N\in \N$ such that every finite-dimensional irreducible unitary representation of $\Gamma$ has dimension at
most $N$.
\end{enumerate}
\end{theo}
\begin{proof}

\medskip
(1) is equivalent to (2) by \cite[\S 13.11.12 p.308]{DIXM}. (6) implies (5) and (5) implies (4) trivially. (4) implies (3) since  $\Gamma$ is residually finite.

We prove that (3) implies (1). By assumption, each $\ell^2(\Gamma/\Gamma_i)$ is the direct sum
of irreducible representations of dimension at most $N$, which are all unitary since the representation of $\Gamma$
on $\ell^2(\Gamma/\Gamma_i)$ is. By considering all $i$, we thus get
an embedding $\Gamma \into \bigoplus_{j \in J} G_j$ with $G_j < U_n :=\mathcal U \left(\C^N\right)$ a finite subgroup of $U_N$.
By Jordan's theorem (see \cite{RAGU}) there exists $q \geq 0$ such that every finite subgroup of $U_N$ has a normal abelian
subgroup of index at most $q$. For all $j\in J$, we let $H_j$ be such a subgroup for $G_j$. We denote by $\varphi_j : G_j \onto G_j/H_j$ the canonical projection
and $\tilde{\varphi}_j : \Gamma \to G_j/H_j$ the induced morphism. By assumption $|G_j/H_j| \leq q$
hence $G_j/H_j < \mathfrak{S}_q$. The set of subgroups $\Ker \tilde{\varphi}_j < \Gamma$ is thus a subset
of $\{ \Ker \psi \ | \ \psi \in \Hom(\Gamma,\mathfrak{S}_q) \}$,
and $\Hom(\Gamma,\mathfrak{S}_q)$ is finite because $\Gamma$ is finitely generated. It follows that
$H = \bigcap_j \Ker \tilde{\varphi}_j$ is a finite index subgroup of $\Gamma$. Moreover, $H$ is abelian
because its image in $\bigoplus_j G_j$ is abelian.

We now prove that (1) implies (6), and let $H \vartriangleleft \Gamma$ be an abelian normal subgroup such that $\Gamma/H$ is finite.
Let $\rho : \Gamma \to \mathcal U \left( V \right)$ with $V$ a finite-dimensional hermitian space be a unitary representation of $\Gamma$.
The restriction of $\rho$ to $H$ is a direct sum $\bigoplus_{\chi \in \Irr(H) } a_{\chi} \chi$ where $\Irr(H)$
denotes the set of irreducible (1-dimensional) unitary representations of $H$, and $a_{\chi}\in \N$ denotes the multiplicity of $\chi$. Letting 
\[V_r := \bigoplus_{a_{\chi} = r}r \chi\]
we have a canonical decomposition 
\[V = \bigoplus_{r \geq 0} V_r.\] 

We first prove that $V = V_r$ for some $r$. Since
$V$ is irreducible, this is equivalent to saying that $\Gamma. V_r \subset V_r$ for all $r$.  Since $V$ is
finite dimensional, this is equivalent (by descending induction on $r$) to saying that $\Gamma. V_r \subset \bigoplus_{s \geq r} V_s$ for all $r$. We prove this last statement, and decompose $V_r = U_1 \oplus \dots \oplus U_m$ with
$\dim U_i = r$ and $h(x) = \chi_i(h) x$ for every $x \in U_i, h \in H$, for the given $\chi_i \in \Irr(H)$ associated
to $i$ (by construction, $i \neq j \Leftrightarrow \chi_i \neq \chi_j$). Let $g \in \Gamma$, and $i \in \{ 1, \dots,m \}$.
Since $H $ is a normal subgroup of $ \Gamma$, the subspace $g.U_i$ is $H$-stable, hence is included in some $V_s$. Moreover,
$H$ acts through $\chi_i^g : h \mapsto \chi_i(g^{-1}hg)$ on $g . U_i$, hence $\chi_i^g \in \Irr(H)$ occurs in $V_s$
with multiplicity at least $r$. It follows that $s \geq r$. 

Thus 
\[V = V_r =  U_1 \oplus \dots \oplus U_m.\] 

Moreover, we proved $g. U_i = U_{\pi(g)(i)}$
for some $\pi(g) \in \mathfrak{S}_m$, and clearly $\pi : \Gamma \to \mathfrak{S}_m$ is a group morphism, which
factors through $\Gamma/H$. By irreducibility
of $V$, $\pi(\Gamma)= \pi(\Gamma/H)$ must be transitive, which implies $m \leq |\Gamma/H|$.

We now introduce $St(U_1) = \{ g \in \Gamma \ | \ g .U_1 = U_1 \}$, and prove that $U_1$ is irreducible
under $St(U_1)$. Let $E \subset U_1$ be non-zero $St(U_1)$-stable subspace. By irreducibility of $V$ under $\Gamma$, we have
$\sum_{g \in \Gamma} g E = V$. But $g E \subset U_{\pi(g)(1)}$ for all $g \in \Gamma$,
so this can be true only if 

\[\sum_{g \in \Gamma \ | \ \pi(g)(1) = 1} g E = U_1 \ \Longleftrightarrow \ E = \sum_{g \in St(U_1)} g E = U_1, \]
which proves that $U_1$ is $St(U_1)$-irreducible. Since $H$ acts by scalars on $U_1$, this last subspace can be considered as an irreducible projective representation of the subgroup $St(U_1)/H$ of $\Gamma/H$. This last group is finite, hence it has a finite number of subgroups, and
these subgroups have a finite number of irreducible projective representations (for instance by
the theory of Schur covers, see \cite{KARPI}). Taking for $e$ the maximal degree of such representations, we get
$r \leq e$ hence $\dim V = m r \leq N = |\Gamma/H| e$, which proves the claim.
\end{proof}

\end{document}